\documentclass{article}

\usepackage{amssymb,amsmath,amsthm}
\usepackage[]{algorithm2e}
\usepackage{tikz}
\newtheorem{definition}{Definition}[section]
\newtheorem{theorem}[definition]{Theorem}
\newtheorem{corollary}[definition]{Corollary}
\newtheorem{example}[definition]{Example}

\newtheorem{lemma}[definition]{Lemma}
\newtheorem{proposition}[definition]{Proposition}

\usepackage{ifxetex}
\ifxetex
  \usepackage{fontspec}
\else
  \usepackage[T1]{fontenc}
  \usepackage[utf8]{inputenc}
  \usepackage{lmodern}
\fi

\title{On the integer $\{k\}$-domination number of circulant graphs}
\date{May 9, 2019}

\author{
Yen-Jen Cheng,
\thanks{Department of Mathematics, National Taiwan Normal University, Taipei 11677, Taiwan.
{\tt Email: yjc7755@gmail.com}}
\and
Hung-Lin Fu,
\thanks{Department of Applied Mathematics, National Chiao Tung University, Hsinchu 30010, Taiwan.
{\tt Email: hlfu@math.nctu.edu.tw}}
\and
Chia-an Liu
\thanks{Department of Mathematical Sciences, University of Delaware Newark, Delaware 19716, U.S.A.
{\tt Email: liuchiaan8@gmail.com}}
}

\begin{document}
\maketitle

\begin{abstract}
Let $G=(V,E)$ be a simple undirected graph. $G$ is a circulant graph defined on $V=\mathbb{Z}_n$ with difference set $D\subseteq \{1,2,\ldots,\lfloor\frac{n}{2}\rfloor\}$ provided two vertices $i$ and $j$ in $\mathbb{Z}_n$ are adjacent if and only if $\min\{|i-j|, n-|i-j|\}\in D$. For convenience, we use $G(n;D)$ to denote such a circulant graph.

A function $f:V(G)\rightarrow\mathbb{N}\cup\{0\}$ is an integer $\{k\}$-domination function if for each $v\in V(G)$, $\sum_{u\in N_G[v]}f(u)\geq k.$ By considering all $\{k\}$-domination functions $f$, the minimum value of $\sum_{v\in V(G)}f(v)$ is the $\{k\}$-domination number of $G$, denoted by $\gamma_k(G)$. In this paper, we prove that if $D=\{1,2,\ldots,t\}$, $1\leq t\leq \frac{n-1}{2}$, then the integer $\{k\}$-domination number of $G(n;D)$ is $\lceil\frac{kn}{2t+1}\rceil$.
\end{abstract}
\bigskip



\section{Introduction and preliminaries}
The study of domination number of a graph $G$ has been around for quite a long time. Due to its importance in applications, there are various versions of extension study, see \cite{hhs98} for reference.

The idea of integer $\{k\}$-domination was proposed by Domke et al. in \cite{dhlf91}. It can be dealt as a labeling problem. The vertices of the graph $G$ are labeled by integers in $\mathbb{N}\cup \{0\}$ such that for each vertex $v$, the total (sum) values in its closed neighborhood $N_G[v]$ must be at least $k$. The problem is asking for finding the minimum total value labeled on $G$. Finally, we say that $f:V(G)\rightarrow \mathbb{N}\cup\{0\}$ is an integer $\{k\}$-domination function if for each $v\in V(G)$, $\sum_{u\in N_G[v]}f(u)\geq k$. Among all such functions $f$, the minimum value of $\sum_{v\in V(G)}f(v)$ is called the integer $\{k\}$-domination number of $G$, denoted by $\gamma_k(G)$.

It is not difficult to see that the original domination number of a graph $G$, $\gamma(G)$, can be recognized as $\gamma_1(G)$ since the vertices with label "1" gives a dominating set. For more information about domination problem, the readers may refer to \cite{cm12,gprt11,g06,hhs98,js13}.
Hence, the integer $\{k\}$-domination problem is also an NP-hard problem. So far, results obtained are all on special classes of graphs, see \cite{bhk06,dhlf91,hl09,k17}.

In this paper, we shall consider the class of circulant graph $G=G(n;D)$ where $D=\{1,2,\ldots,t\}$, $1\leq t\leq \frac{n-1}{2}$, i.e., $V(G)=\mathbb{Z}_n$ and two vertices $i$ and $j$ are adjacent if and only if $d(i,j):=\min\{|i-j|,n-|i-j|\}\in D$. Since $D=\{1,2,\ldots,t\}$, $G(n;D)$ is exactly the power graph $C_n^t$ where $C_n$ is a cycle of order $n$.

The following results are obtained by Lin \cite{l18}. For clearness, we also outline its proof in which basic linear algebra is applied.
\begin{proposition}[\cite{l18}]\label{prop1.1}
Let $G$ be the circulant graph $G(n;D)$ where $D=\{1,2,\ldots,t\}$. Then, $\gamma_k(G)\geq \lceil\frac{kn}{2t+1}\rceil.$
\end{proposition}
\begin{proof}
Let $A$ be the adjacency matrix of $G$ and $f$ be an $\{k\}$-domination function of $G$. Let ${\bf 1}_n$ denote the all 1 column vector of length $n$. Then, we have
$$(2t+1)\sum_{v\in V(G)}f(v)=(f(v_1),f(v_2),\ldots,f(v_n))(A+I_n){\bf 1}_n\geq {\bf 1}_n^T\cdot k\cdot{\bf 1}_n=nk,$$
which implies the inequality.
\end{proof}

By the aid of an algorithm, Lin was able to show the following.
\begin{proposition}[\cite{l18}]\label{prop1.2}
For $t\leq 5$, $\gamma_k(G(n;\{1,2,\ldots,t\}))=\lceil\frac{kn}{2t+1}\rceil$.
\end{proposition}
But, for larger $t$, it remains unsettled. Our main result of this paper shows that the equality holds for all $1\leq t\leq \frac{n-1}{2}$.

\section{The main result}
By Proposition \ref{prop1.1}, in order to determine $\gamma_k(G)$, it suffices to show that $\gamma_k(G)\leq \lceil\frac{nk}{2t+1}\rceil$. That is, we need a proper distribution of values for $f(v_1),f(v_2),\ldots,f(v_n)$ such that for each $v_i$, $\sum_{u\in N_G[v_i]}f(u)\geq k$ and $\sum_{i=1}^n f(v_i)\leq \lceil\frac{nk}{2t+1}\rceil$. Since we are dealing with circulant graphs, $\sum_{u\in N_G[v_i]}f(u)$ is in fact the sum of $2t+1$ consecutive labels assigned to the circle $C_n=(v_1,v_2,\ldots,v_n)$. Therefore, we turn our focus on providing suitable labels to meet the condition.

For example, let $n=8$ and $t=2$. Then, the following labeling of $(v_1,v_2,\ldots,v_8)$, $(x_4,x_3,x_1,x_2,x_0,x_4,x_3,x_1)$ will satisfy the requirement, where $x_i=\lfloor\frac{k+i}{5}\rfloor$, $i=0,1,2,3,4.$

We are considering $1\leq t\leq \frac{n-1}{2}$ in what follows. First, we need an estimaton of the sum of rational numbers which take its floor or ceiling values.
\begin{lemma}\label{lem2.1}
For positive integers $a,b$ and nonnegative integer $k$, we have the following.
\begin{enumerate}
\item[(1)] $\lfloor\frac{k+\lfloor x\rfloor}{a}\rfloor=\lfloor \frac{k+x}{a} \rfloor$
for any real number $x,$
\item[(2)] $k=\sum_{i=0}^{a-1}\lfloor \frac{k+i}{a} \rfloor$, and
\item[(3)] $\lceil\frac{ak}{b}\rceil = \sum_{i=1}^{a}\lfloor\frac{k+\lceil ib/a \rceil-1}{b}\rfloor$.
\end{enumerate}
\end{lemma}
\begin{proof}
(1) and (2) are easy to check, we prove (3).
\begin{eqnarray}
\lceil\frac{ak}{b}\rceil &=& \sum_{i=0}^{a-1}\lfloor \frac{\lceil ak/b \rceil+i}{a} \rfloor
\nonumber   \\
&=& \sum_{i=0}^{a-1}\lfloor \frac{\lfloor (ak+b-1)/b \rfloor+i}{a} \rfloor
= \sum_{i=0}^{a-1}\lfloor \frac{(ak+b-1)/b+i}{a} \rfloor
\nonumber   \\
&=& \sum_{i=0}^{a-1}\lfloor \frac{k+((i+1)b-1)/a}{b} \rfloor
= \sum_{i=0}^{a-1}\lfloor \frac{k+\lfloor((i+1)b-1)/a\rfloor}{b} \rfloor
\nonumber   \\
&=& \sum_{i=0}^{a-1}\lfloor \frac{k+\lceil((i+1)b-1-a+1)/a\rceil}{b} \rfloor
= \sum_{i=0}^{a-1}\lfloor\frac{k+\lceil (i+1)b/a \rceil-1}{b}\rfloor.
\nonumber
\end{eqnarray}
\end{proof}
\medskip

Since the variables $a,b$ and $k$ are all integers, the uniqueness of the formula in Lemma~\ref{lem2.1}(3) can be confirmed.
\begin{corollary}\label{cor_ceil_floor}
{If integers $0\leq s_0\leq s_1\leq\cdots\leq s_{a-1}<b$ satisfy
$$\lceil\frac{ak}{b}\rceil = \sum_{i=0}^{a-1}\lfloor\frac{k+s_i}{b}\rfloor$$
for positive integers $a,b$ and nonnegative integer $k$, then $s_i=\lceil (i+1)b/a \rceil-1$ for $i=0,1,\ldots,a-1.$
}
\qed
\end{corollary}
\medskip

According to the $s_i$'s given above, we split $[b]:=\{0,1,\ldots,b-1\}$ into subintervals with maximal elements $s_i$'s.
For positive integers $a<b,$ let $[b]$ be partitioned into $a$ subsets such that
$$S_i=\left\{\lceil \frac{bi}{a}\rceil,\lceil \frac{bi}{a}\rceil+1,\ldots,\lceil \frac{b(i+1)}{a}\rceil-1\right\}$$
for $i=0,1,\ldots,a-1.$ It is clear that {$||S_i|-|S_j||$}$\leq 1$ for all $i,j.$
We analyze the subsets containing more elements in the following.
\begin{lemma}       \label{lemma_set_q+1}
Let $q$ and $r$ be the quotient and remainder of $b$ divided by $a,$ respectively. Then the cardinality $|S_{i}|=q+1$ if $i=\lfloor \frac{aj}{r}\rfloor$ for $j=0,1,\ldots,r-1.$
\end{lemma}
\begin{proof}
By definition, $|S_i|=\lceil(i+1)b/a\rceil-\lceil ib/a\rceil=q+\lceil(i+1)r/a\rceil-\lceil ir/a\rceil.$
Therefore, $|S_i|=q+1$ if and only if there exists some integer $0\leq j\leq r-1$ such that $ir/a\leq j< (i+1)r/a.$
The above inequality can be rewritten as $i\leq ja/r {<} i+1,$ and hence $i=\lfloor ja/r\rfloor.$
\end{proof}

\medskip

\begin{example}
Let $a=3,$ $b=8,$ and $r=2$ be the remainder of $b$ divided by $a.$ Then
\begin{equation}    \label{eq_exam1}
\lceil\frac{3k}{8}\rceil=\sum_{i=1}^{3}\lfloor\frac{k+\lceil 8i/3 \rceil-1}{8}\rfloor
=\lfloor\frac{k+2}{8}\rfloor+\lfloor\frac{k+5}{8}\rfloor+\lfloor\frac{k+7}{8}\rfloor,
\end{equation}
and $[8]=\{0,1,\ldots,7\}$ can be partitioned into 3 subsets such that
$$S_0=\{0,1,2\},~S_1=\{3,4,5\}~\text{and}~S_2=\{6,7\},$$
where the subsets numbered with $\lfloor \frac{aj}{r}\rfloor=0$ and $1$ as $j=0$ and $1,$ respectively, have more than $1$ elements.
{Note that the maximal elements $2,5,7$ of subsets $S_i$'s are the integers in~\eqref{eq_exam1} that construct $\lceil\frac{3k}{8}\rceil$.}
\end{example}

\medskip

Additionally, we need a result of the comparison between two sequences.
For two real finite non-decreasing sequences $A=(a_i),A'=(a'_i)$ of the same length $n$, we say that $A\leq A'$ if $a_i\leq a'_i$ for $i=0,1,\ldots,n-1.$
\begin{lemma}       \label{lemma_seq_ineq}
Let $A$ and $A'$ be two subsequences of a real finite non-decreasing sequence $B$ which have equal length $0<|A|=|A'|<|B|$. Then $A\leq A'$ if and only if $B\setminus A'\leq B\setminus A.$
\end{lemma}
\begin{proof}
Because of the symmetry, we prove $A\leq A'$ implies $B\setminus A'\leq B\setminus A$ by induction on $|A|$ in the following.
It is clearly true when $|A|=1.$ Suppose the statement is correct for $|A|<m<|B|.$ Assume that $A=(a_i)_{i=0}^{m-1}$ and $A'=(a'_i)_{i=0}^{m-1}$ satisfying $A\leq A'.$
From induction hypothesis, $B\setminus (a'_i)_{i=1}^{m-1}\leq B\setminus (a_i)_{i=1}^{m-1}.$
It is clear that the non-decreasing sequence obtained by exchanging an entry $a$ of the original sequence into $\tilde{a}\geq a$ (and inserting $\tilde{a}$ to the appropriate position) is not less than the original sequence. Thus, we have
$$B\setminus A' \leq B\setminus \widetilde{A} \leq B\setminus A,$$
where $\widetilde{A}$ is obtained from $A$ by deleting $a_0$ and adding $a'_0.$ The result follows.
\end{proof}

\bigskip

Now, we are ready to find the desired integer $\{k\}$-domination function $f$.
Let $[a]:=\{0,1,\ldots,a-1\}$ for each positive integer $a.$ For a sequence $A$ of length $a,$ let the entries of $A$ indexed by $[a]$ and $A(i)$ be the $i$-th entry of $A.$ For $0\leq i<j\leq a,$ the subsequence $A[i:j]:=[A(i),A(i+1),\ldots,A(j-1)].$ If $A$ is a permutation of $[a],$ then the complement of $A$ is a sequence $\overline{A}$ of length $a$ defined as $\overline{A}(i)=a-1-A(i)$ for $0\leq i\leq a-1.$ The \emph{concatenation} $A\circ B$ of two sequences $A$ and $B$ of lengths $a$ and $b,$ respectively, is a sequence of length $a+b$ obtained by attaching $B$ to $A$ defined as $(A\circ B)(i)=A(i)$ if $0\leq i\leq a-1$ and $(A\circ B)(j)=B(j-a)$ if $a\leq j\leq a+b-1.$

\medskip

Let $A$ be a permutation of $[a]$ and thus a sequence of length $a$. For positive integers $a<b,$ we call $B$ the \emph{extension sequence} of the pair $(A,b)$ if $B$ is a permutation of $[b]$ satisfying $B(i)<B(j)$ if and only if $A(i_0)<A(j_0)$ or $A(i_0)=A(j_0)$ with $i<j,$ where $i_0$ and $j_0$ are the remainders of $i$ and $j$ divided by $a,$ respectively.
For example, when $(a,b)=(3,7)$ and $A=[0,1,2],$ the extension sequence of $(A,b)$ is $B=[0,3,5,1,4,6,2],$ which is attained by extending $A$ to the sequence $[0,1,2,0,1,2,0]$ of length $7$ and renumbering it with $0,1,\ldots,6.$

\medskip

A permutation $A$ of $[a]$ is said to be \emph{nice} corresponding to some $b>a$ with $a\nmid b$ if
\begin{equation}    \label{eq_nice_1}
A(i)<A(i+r)~~~\text{for}~0\leq i\leq a-r-1
\end{equation}
and
\begin{equation}    \label{eq_nice_2}
A(j)<A(j-a+r)~~~\text{for}~a-r\leq j\leq a-d-1,
\end{equation}
where $r$ is the remainder of $b$ divided by $a$ and $d=\gcd(a,b).$ Note that if $r=d$ then the condition \eqref{eq_nice_2} can be ignored. For example, $[1,3,0,2,4]$ is nice corresponding to $8$ (or any larger integer congruent to $3$ modulo $5$) and $[4,1,6,3,0,5,2,7]$ is nice corresponding to $13.$

\medskip

The following properties will carry out the recursive constructions.
\begin{proposition}       \label{prop_nice_perm}
Suppose that $R$ is a nice permutation of $[r]$ corresponding to some $a>r$ with $r\nmid a.$ Let $\overline{R}$ be the complement of $R.$ Then the extension sequence of $(\overline{R},a)$ is also nice corresponding to some $b>a$ with $b\equiv r~(\text{mod}~a).$
\end{proposition}
\begin{proof}
Let $A$ be the extension sequence of $(\overline{R},a).$ Note that $A(i)<A(i+r)$ for $0\leq i\leq a-r-1$ can be verified directly by the definition of extension sequences. Assume that $s$ is the remainder of $a$ divided by $r.$
It's left to consider the case $a-r\leq j\leq a-d-1$ where $r$ is the remainder of $b$ divided by $a$ and $d=\gcd(a,b).$ Assume that $s$ is the remainder of $a$ divided by $r.$ By Euclidean algorithm, $s\geq d.$

\smallskip

\noindent\textbf{Case 1: $a-r\leq j\leq a-s-1.$}\\
In order to show $A(j)<A(j-a+r),$ we observe that the remainders of $j$ and $j-a+r$ divided by $r$ are $j'+s$ and $j'$, respectively, where $j'=j-a+r.$ Moreover, since $R$ is nice, we have $\overline{R}(i+s)<\overline{R}(i)$ for $0\leq i\leq r-s-1.$ The result is straightforward by the definition of extension sequences.

\smallskip

\noindent\textbf{Case 2: $a-s\leq j\leq a-d-1.$}\\
In this case, the remainders of $j$ and $j-a+r$ divided by $r$ become $j'-r+s$ and $j'$, respectively, where $j'=j-a+r.$ Once again, since $R$ is nice, $\overline{R}(i-r+s)<\overline{R}(i)$ for $r-s\leq i\leq r-d-1.$ We have the proof.
\end{proof}

\medskip

\begin{proposition}     \label{prop_latter_B}
Let positive integers $a<b$ with $r>0$ the remainder of $b$ divided by $a$ and $R$ a permutation of $[r]$ with complement $\overline{R}.$
If the extension sequence $A$ of $(R,a)$ satisfies
$$\lceil\frac{rk}{a}\rceil = \sum_{i=a-r}^{a-1} \lfloor\frac{k+A(i)}{a}\rfloor,$$
then the extension sequence $B$ of $(A',b)$ satisfies
$$\lceil\frac{ak}{b}\rceil = \sum_{i=b-a}^{b-1} \lfloor\frac{k+B(i)}{b}\rfloor$$
where $A'$ is the extension sequence of $(\overline{R},a).$
\end{proposition}
\begin{proof}
By {Corollary \ref{cor_ceil_floor}},
$$\left\{A(i)\mid a-r\leq i\leq a-1\right\}=\left\{\lceil \frac{aj}{r}\rceil-1 \mid 1\leq j\leq r\right\}.$$
Let $q$ and $s$ be the quotient and remainder of $a$ divided by $r,$ respectively. Claim that the set of ${A'}[0:r]$ equals the set of $a-1-A[a-r:a-1].$ It is clear for $s=0.$ If $s>0,$ we have
$$A(i) = a-1-{A'}(a-s+i)~~~\text{for}~i=0,1,\ldots,s-1,$$
since entries in $A$ that larger than $A(i)$ become smaller than ${A'}(a-s+i)$ in ${A'},$ and vice versa. Therefore, the set of $A[s+jr:s+(j+1)r]$ equals to the set of $a-1-{A'}[a-s-(j+1)r:a-s-jr]$ for $j=0,1,\ldots,q.$ The claim follows by taking $j=q-1$.
Moreover, since
$$a-1-(\lceil\frac{ai}{r}\rceil-1)=a+\lfloor\frac{-ai}{r}\rfloor=\lfloor\frac{a(r-i)}{r}\rfloor$$
for $1\leq i\leq r,$ we have
$$\left\{{A'}(i)\mid 0\leq i\leq r-1\right\}=\left\{\lfloor \frac{aj}{r}\rfloor \mid 0\leq j\leq r-1\right\}$$
which exactly indicates the indices of subsets defined in Lemma~\ref{lemma_set_q+1}.
Hence the set of $B[b-a:b]$ gives the maximal elements in each of the subsets $S_0,S_1,\ldots,S_{a-1},$ and this fact completes the proof.
\end{proof}

\medskip

For each pair of positive integers $(a,b)$ with $a<b,$ define two codes ${C_1}$ and ${C_2}$ as follows. If $a$ divides $b,$ then
$$C_1(a,b):=[0,1,\ldots,a-1].$$
If the remainder $r$ of $b$ divided by $a$ is positive, then
$$C_1(a,b):=~\text{the extension sequence of}~ (\overline{C_1(r,a)},a)$$
where $\overline{C_1}(r,a)$ is the complement of ${C_1}(r,a).$ Now $C_2$ can be constructed subsequently. Let $C_2(a,b)$ be the extension sequence of $(C_1(a,b),b).$ It is clear that $C_1(a,b)$ and $C_2(a,b)$ are permutations of $[a]$ and $[b],$ respectively. Suppose that each entry $\alpha$ in $C_2(a,b)$ is corresponding to $\lfloor \frac{k+\alpha}{b}\rfloor.$ Then the following result can be obtained by proving that
$$C:=B[b-a:b]\circ\underbrace{B\circ B\circ\cdots\circ B}_q$$
is a feasible distribution of the circulant graph $G,$ where $b=2t+1,$ $n=qb+a,$ and $B=C_2(a,b).$
\begin{theorem}     \label{thm_main}
$$\gamma_k(G) = \lceil \frac{kn}{2t+1} \rceil.$$
\end{theorem}
\begin{proof}
By Proposition \ref{prop1.1}, it suffices to show that $\gamma_k\leq \lceil\frac{kn}{2t+1}\rceil$. Let $G=G(n;\{1,2,\ldots,t\})$, $n=qb+a$ and $b=2t+1$. First, we construct $B[b-a:b]$.
If $a$ divides $b$ such that $b=\ell a,$ then
$$B[b-a:b]=[\ell-1,2\ell-1,\ldots,a\ell-1]$$
which collects the numbers $\lceil ib/a\rceil-1$ for $1\leq i\leq a$ given in Lemma~\ref{lem2.1}. Therefore, any substring of $C$ of length $a$ is not larger than $B[b-a:b].$ By Lemma~\ref{lemma_seq_ineq}, 
every length $b$ string of $B\circ B[b-a:b]$ or $B[b-a:b]\circ B$ is not less than $[0,1,\ldots,b-1],$ so does $C.$
Furthermore, the sequence $[0,1,\ldots,b-1]$ is of sum
$$\sum_{i=0}^{b-1}\lfloor\frac{k+i}{b}\rfloor=k,$$
which confirms the case for $a$ divides $b.$

\medskip

On the other hand, let $a>\gcd(a,b)$ and $r$ be the remainder of $b$ divided by $a.$ Since the initial case is examined above, by Proposition~\ref{prop_latter_B}, $B[b-a:b]$ collects the elements $\{\lceil ib/a\rceil-1\}_{i=1}^a.$ For the initial case, if $r$ divides $a$ then $C_1(r,a)=[0,1,\ldots,r-1]$ and it is easy to check that $C_1(a,b)$ is nice. Moreover, by Proposition~\ref{prop_nice_perm}, $C_1(a,b)$ is always nice, and hence
$$C[i:i+a]\leq B[b-a:b]~~~\text{for}~0\leq i\leq a-d-1,$$
where $d=\gcd(a,b).$ Moreover, we also have $C[i:i+a]\leq B[b-a:b]$ for $a-d\leq i\leq qb-1$ immediately from the construction of $C.$ The result follows.
\end{proof}

\medskip

\begin{example}
Assume that $G(n;D)$ is a circulant graph on $n=8$ vertices with $D=\{1,2\}$ (i.e., $t=2$). Let $b=2t+1=5$ and $a=3$ be the remainder of $n$ divided by $b$. First of all, we obtain $C_1(3,5)$ by the process of Euclidean algorithm. Since the initial condition $C_1(1,2)=[0],$ $C_1(2,3)=[0,1]$ is directly the extension code of $[0].$ Next, the complement of $C_1(2,3)$ is $\overline{C_1(2,3)}=[1,0]$. Thus, $C_1(3,5)=[1,0,2],$ the extension code of $([1,0],3).$ Hence,
$$C_2(3,5)=[2,0,4,3,1],$$
the extension code of $(C_1(3,5),5).$ Attach the last $3$ entries in front of $C_2(3,5),$ we attain the distribution $[4,3,1,2,0,4,3,1].$
Then the circular sequence $f(v):~v\in V(G)$ is given by
$$(\lfloor\frac{k+4}{5}\rfloor,\lfloor\frac{k+3}{5}\rfloor,\lfloor\frac{k+1}{5}\rfloor,\lfloor\frac{k+2}{5}\rfloor,
\lfloor\frac{k}{5}\rfloor,\lfloor\frac{k+4}{5}\rfloor,\lfloor\frac{k+3}{5}\rfloor,\lfloor\frac{k+1}{5}\rfloor)$$
which satisfies $\sum_{v\in V(G)}f(v)=\lceil 8k/5\rceil$ and $\sum_{u\in N_G[v]}f(u)\geq k$ for each $v\in V(G).$
\end{example}

\section{Concluding remark}
We remark finally that the construction of code $C_2$ can be obtained by giving an algorithm with inputs $a$ and $b$.
\medskip

\begin{algorithm}[H]    \label{algorithm_1}
 \KwData{Positive integers $a<b.$}
 \KwResult{$C_2(a,b).$}
%
 \SetKwFunction{FMain}{Main}
 \SetKwFunction{C}{$C_1$}
%

 \C{$a$,$b$}{
    \If{$a=\gcd(a,b)$}{
        \KwRet $[0,1,\ldots,a-1]$\;
    }
    \Else{
        $r\leftarrow$ the remainder of $b$ divided by $a$\;
        $R\leftarrow C_1(r,a)$\;
        \KwRet the extension sequence of $(\overline{R},a)$\;
    }
 }
 \FMain{$a$,$b$}{
    \KwRet the extension sequence of $(C_1(a,b),b)$\;
 }
\end{algorithm}

\medskip



\begin{thebibliography}{88}
%
\bibitem{bhk06}
B. Bre\v{s}ar, M. A. Henning and S. Klav\v{z}ar, On integer domination in graphs and vizing-like problems,
{\it Taiwanese J. Math.} 10(2006) 1317-1328.

\bibitem{cm12}
T. T. Chelvam and S. Mutharasu, Bounds for domination parameter in circulant graphs,
{\it Advanced Studies in Contemporary Mathematics (Kyungshang)} 22(4)(2012) 525-529.

\bibitem{dhlf91}
G. Domke, S. T. Hedetniemi, R. C. Laskar and G. Fricke, Relationships between integer and fractional parameters of graphs,
{\it Graph Theory, Combinatorics, and applications} 1(1991) 371-387.

\bibitem{gprt11}
D. Gon\c{c}alves, A. Pinlou, M. Rao and S. Thomass\'{e}, The domination number of grids,
{\it SIAM J. Discrete Math.} 25(3)(2011) 1443-1453.

\bibitem{g06}
R. M. Gray, Toeplitz and circulant matrix: a review,
{\it Fundation and Trends in Communication Theory} 2(2006) 155-239.

\bibitem{hhs98}
T. W. Haynes and S. T. Hedetniemi, P. J. Slater, Domination in graphs: advanced topics,
Marcel Deliker, New York (1998).

\bibitem{hl09}
X. Hou and Y. Lu, On the $\{k\}$-domination number of Cartesian products of graphs,
{\it Discrete Math.} 309(2009) 3413-3419.

\bibitem{js13}
N. John and S. Suen, Graph products and integer domination,
{\it Discrete Math.} 313(2013) 217-224.

\bibitem{k17}
Y.-T. Kuan, A study of integer domination number,
M. S. Thesis, National Chian Tung University (2017).


\bibitem{l18}
X. Lin, Integer $\{k\}$-domination number of circulant graphs,
M. S. thesis, National Chiao Tung University (2018).

\end{thebibliography}
\end{document}